\documentclass[oneside]{amsart}
\usepackage{amsmath, amssymb}
\usepackage{amsthm}
\newtheorem{theorem}{Theorem}[section]

\newtheorem{lemma}[theorem]{Lemma}
\newtheorem{corollary}[theorem]{Corollary}

\theoremstyle{definition}

\begin{document}

\title[Toral or non locally connected minimal sets]
{Toral or non locally connected minimal sets for  
$R$-closed surface homeomorphisms}
\author{Tomoo Yokoyama}
\date{\today}
\email{yokoyama@math.sci.hokudai.ac.jp}

\thanks{The author is partially supported 
by the JST CREST Program at Department of Mathematics,  
Hokkaido University.}

\maketitle

\begin{abstract} 
Let $M$ be an orientable connected closed surface 
and 
$f$ be an $R$-closed homeomorphism on $M$ 
which is isotopic to identity. 
Then 
the suspension of $f$ satisfies 
one of the following condition:  
1) the closure of each element of it is toral. 
2) there is a minimal set which is not locally connected. 
Moreover,  
we show that 
any positive iteration of an $R$-closed 
homeomorphism on a compact metrizable space is 
$R$-closed. 
\end{abstract}


\section{Preliminaries}
In this paper, 
we will show that 
if $f$ is a nontrivial $R$-closed homeomorphism on $M$ 
which is not periodic but isotopic to identity, 
then 
either a) $f$ is ``an irrational rotation'' 
or b) there is a minimal set which is not locally connected. 
Moreover 
if $M$ has genus $\geq 2$, then b) holds.
Taking suspensions, 
we show that 
a) implies that 
each orbit closure is a torus. 
In addition, 
we show that 
any positive iteration of an $R$-closed 
homeomorphism on a compact metrizable space is 
$R$-closed.

For a subset $U$ of a topological space, 
$U$  
is locally connected 
if every point of $U$ admits a neighborhood basis 
consisting of open connected subsets. 
For a (binary) relation $E$ on a set $X$ 
(i.e. a subset of $X \times X$), 
let $E(x) := \{ y \in X \mid (x,  y) \in  E\}$ 
for an element $x$ of $X$. 
For a subset $A$ of $X$, 
we say that 
$A$ is $E$-saturated if 
$A = \cup_{x \in A} E(x)$. 
Also 
$E$ define the relation $\hat{E}$ on 
$X$ with $\hat{E} (x) = \overline{E(x)}$. 
Recall that 
$E$ is pointwise almost periodic if 
$\hat{E}$ is an equivalence relation 
and 
$E$ is $R$-closed if 
$\hat{E}$ is closed. 
For an equivalence relation $E$, 
the collection of equivalence classes 
$\{E(x) \mid x \in X\}$ is a decomposition of X, 
denoted by $\mathcal{F} _E$. 
By a decomposition, 
we mean a family $\mathcal{F} $ of pairwise disjoint subsets of a set $X$ 
such that $X = \sqcup \mathcal{F}$. 
For a homeomorphism $f$ on $X$, 
let $E_f$ be the equivalent relation by 
$E_f (x) := \{ f^k(x) \mid k \in \mathbb{Z}  \}$. 
$f$ is said to be $R$-closed if 
$E_f$ is $R$-closed. 
Then $f$ is $R$-closed if and only if 
 $R := \{ (x, y) \mid y \in \overline{O_f(x)} \}$ is closed. 
Note 
that 
$f$ on a locally compact Haudorff space  is pointwise almost periodic 
if and only if 
$\hat{E}_f$ is an equivalent relation 
(cf. Theorem 4.10 \cite{GH}).  
We call that 
an equivalence relation $E$ is $L$-stable 
if 
for an element $x$ of $X$  
and 
for any open neighborhood 
$U $ of $\hat{E} (x)$,  
there is a $E$-saturated open neighborhood  $V$ of $\hat{E} (x)$ 
contained in $U$. 
In \cite{ES}, 
they show the following: 
If a continuous mapping $f$ of a topological space $X$ 
in itself 
is either pointwise recurrent or 
pointwise almost periodic,  
then so is $f^k$ for each positive integer $k$. 
In general cases, 
see Theorem 2.24, 4.04, and 7.04 \cite{GH}. 
We show 
the following key lemma 
which is 
an $R$-closed version of this fact  
on a compact metrizable space.

\begin{lemma}\label{lem:001}
Let $f$ an homeomorphism on a 
compact metrizable space $X$. 
If $f$ is $R$-closed, 
then 
so is $f^n$ for any $n \in \mathbb{Z} _{>0}$.  
\end{lemma}

\begin{proof}
Put 
$E := E_{f}$ 
and $E^n := E_{f^n}$. 
By Corollary 1.5 \cite{Y}, 
we have that 
$\hat{E}$ is an equivalence relation 
 and so $f$ are pointwise almost periodic. 
Since $f$ is pointwise almost periodic, 
by Theorem 1\cite{ES}, 
we have that 
$f^n$ is also pointwise almost periodic.  
Then $\hat{E}^n$ is an equivalence relation. 
By Corollary 3.6\cite{Y}, 
$E$ is $L$-stable 
and 
it suffices to show that 
$E^n$ is $L$-stable. 
Note that 
$E^n (x) \subseteq E (x) $ and so 
$\hat{E}^n (x)  \subseteq \hat{E} (x) $. 
For $x \in X$ with $\hat{E}^n (x) = \hat{E} (x)$ and 
for any open neighborhood 
$U $ of $\hat{E} (x) = \hat{E}^n (x)$, 
since $E$ is $L$-stable, 
there is a 
$E$-saturated open neighborhood  $V$ of $\hat{E} (x)$ 
contained in $U$. 
Since $E^n (x)  \subseteq E (x) $, 
we have that 
$V$ is also 
a ${E^n}$-saturated open neighborhood  $V$ of $\hat{E} (x)$. 
Fix any $x \in X$ with $\hat{E} (x) \neq \hat{E}^n (x)$. 
Put $\{ \hat{E}_1, \dots , \hat{E}_k \} := 
\{ \hat{E} (f^k (x)) \mid k = 0, 1, \dots , n-1 \}$ 
such that 
$\hat{E}_1 = \hat{E}^n (x)$ and 
$\hat{E}_i \cap \hat{E}_j = \emptyset$
 for any $i \neq j \in \{1, \dots , k\}$. 
Let $\hat{E}' = \hat{E}_2 \sqcup \dots \sqcup \hat{E}_k$. 
Then 
$\hat{E}_1$ and $\hat{E}'$ are closed
and 
$\hat{E} 
= \hat{E}_1 \sqcup \cdots \sqcup \hat{E}_k 
= \hat{E}_1 \sqcup \hat{E}'$. 
For any sufficiently small $\varepsilon  > 0$, 
let 
$U_{1, \varepsilon } = B_{\varepsilon } (\hat{E}_1)$ 
(resp. $U_{\varepsilon }' = B_{\varepsilon } (\hat{E}')$) 
be the open $\varepsilon $-ball of $\hat{E}_1$ 
(resp. $\hat{E}'$). 
Since $\varepsilon $ is small and 
$X$ is normal,  
we obtain 
$U_{1, \varepsilon } \cap U_{\varepsilon }' = \emptyset$,  
$\overline{U_{1, \varepsilon /2}} \subseteq U_{1, \varepsilon }$,  
and 
$\overline{U_{\varepsilon/2}'} \subseteq U_{\varepsilon }'$. 
Since $E$ is $L$-stable, 
there are neighborhoods $V_{1, \varepsilon } \subseteq U_{1, \varepsilon/2}$ 
(resp. $V'_{\varepsilon } \subseteq U_{\varepsilon/2}'$) of $\hat{E}_1$ 
(resp. $\hat{E}'$) 
such that 
$V_{1, \varepsilon } \sqcup V'_{\varepsilon }$ is an $E$-saturated neighborhood of $\hat{E} (x)$. 
Since $\hat{E}_1$ and $\hat{E}'$ 
are $f^n$-invariant and compact, 
there is a small $\delta > 0$ such that 
$f^n(V_{1, \delta}) \subseteq U_{1, \varepsilon }$ and 
$f^n(V'_{\delta}) \subseteq U'_{\varepsilon }$. 
Since $V_{1, \delta} \sqcup V'_{\delta}$ is 
$f^n$-invariant and 
$U_{1, \varepsilon } \cap U'_{\varepsilon } = \emptyset$, 
we obtain 
$
V_{1, \delta} \sqcup V'_{\delta} 
= 
f^n(V_{1, \delta} \sqcup V'_{\delta}) 
= 
f^n(V_{1, \delta}) \sqcup f^n(V'_{\delta}) 
$, 
$f^n(V_{1, \delta}) \cap V'_{\delta} = \emptyset$, 
and 
$f^n(V'_{\delta}) \cap V_{1, \delta} = \emptyset$. 
Hence 
$
V_{1, \delta} = f^n(V_{1, \delta})$
and 
$V'_{\delta} = f^n(V'_{\delta})$. 
This implies that 
$V_{1, \delta} $
is an 
${E^n}$-saturated neighborhood of $\hat{E}_1 = \hat{E}^n (x)$ 
with 
$V_{1, \delta} \subseteq U_{1, \varepsilon } 
= B_{\varepsilon }(\hat{E}_1) 
= B_{\varepsilon }(\hat{E}^n ( x))$. 
\end{proof}

Note this lemma is not true for compact $T_1$ spaces. 
(e.g. 
a homeomorphism $f$ on  
a non-Hausdorff 1-manifold 
$X = \{ 0_-, 0_+ \} \sqcup ]0, 1]$   
by $f(0_{\pm}) = 0_{\mp}$ and 
$f|_{]0,1]} = \mathop{\mathrm{id}}$).  

%

\section{Main results}

From now on, 
let 
$M$ be an orientable connected closed surface
 and 
$f$ a nontrivial $R$-closed homeomorphism on $M$ 
which is not periodic but isotopic to identity. 
%
We call that 
$f$ on $S^2$ is a topological irrational rotation 
if there is an irrational number $\theta_0 \in  \mathbb{R} - \mathbb{Q} $ 
such that 
$f$ is topologically conjugate to 
a map on a unit sphere in $\mathbb{R} ^3$ with 
the Cylindrical Polar Coordinates by 
$(\rho, \theta, z) \to (\rho, \theta + \theta_0, z) 
\in \mathbb{R} _{\geq 0} \times S^1 \times \mathbb{R} $.  
Also 
$f$ on $\mathbb{T}^2$ is a topological irrational rotation 
if there is an irrational number $\theta_0 \in  \mathbb{R} - \mathbb{Q} $ 
such that 
some positive iteration of 
$f$ is topologically conjugate to 
a map $S^1 \times S^1 \to S^1 \times S^1$ by 
$(\theta, \varphi) \to (\theta + \theta_0, \varphi)$. 

\begin{lemma}\label{lem:00}
If every minimal set is locally connected, 
then 
$f$ is a topological irrational rotation on 
$M = S^2$ or $\mathbb{T}^2$. 
\end{lemma}

\begin{proof} 
By Theorem 1 and Theorem 2 \cite{BNW}, 
since $f$ is pointwise almost periodic,  
every orbit closure is a finite subset or 
a finite disjoint union of simple closed curves. 
We will show that 
there is a finite disjoint union of simple closed curves.  
Otherwise $f$ is pointwise periodic. 
By \cite{M}, 
we have $f$  is periodic, which contradicts.    
By Lemma \ref{lem:001}, 
we have that 
$f^n$ is also $R$-closed for any $n \in \mathbb{Z} _{>0}$. 
Hence there is an positive integer $n$ such that 
$f^n$ has a simple closed curve as a minimal set. 
Then 
Theorem 2.4 \cite{Y2} implies that 
$M$ is either $\mathbb{T}^2$ or $S^2$. 
By Corollary 2.5 \cite{Y2}, 
if $M = S^2$, then $f$ has a null homotopic circle and so $n = 1$.  
Suppose $M = \mathbb{T}^2$ (resp. $S^2$). 
By Theorem 2.4\cite{Y2}, 
the set 
$\mathcal{F} _{\hat{E}_{f^n}}$ of orbits closures 
consists of essential circles 
(resp. two singular points and other circles). 
Fix any $x \in M$. 
Then $A := \mathbb{T}^2 - \hat{E}_{f^n} (x)$ 
(resp. $A := S^2 - \mathrm{Sing}({f^n})$) 
is an open annulus. 
By Lemma 1.5 
(resp. the proof of Lemma 2.1) \cite{Y2}, 
we have that 
the restriction $f^n|_A$ to 
the open annulus $A$ is 
an irrational rotation. 
\end{proof}

This implies  our main results. 

\begin{theorem}
Let $M$ be an orientable connected closed surface 
and 
$f$ be a nontrivial $R$-closed homeomorphism on $M$ 
which is not periodic but isotopic to identity. 
Then one of the following holds: 
\\
1) $f$ is a topological irrational rotation. 
\\ 
2) there is a minimal set which is not locally connected. 
\\
Moreover 
2) holds when $M$ has genus $\geq 2$.
\end{theorem}

Taking a suspension, 
we have a following corollary. 

\begin{corollary}
Let $M$ be an orientable connected closed surface 
and 
$f$ be an $R$-closed homeomorphism on $M$ 
which is isotopic to identity. 
Then 
the suspension of $f$ satisfies 
one of the following condition:  
\\
1) the closure of each element of it is  toral. 
\\
2) there is a minimal set which is not locally connected. 
\end{corollary}

%
%
%
%
%
%
%
%
%
%
%
%


\end{document}